\documentclass{article}
\usepackage[utf8]{inputenc}
\usepackage{titling}
\usepackage{amsmath,amssymb,url,color,amsfonts, amsthm}
\usepackage{authblk,abstract}
\usepackage{comment}
\usepackage{tikz}
\usepackage{stmaryrd}
\usepackage{mathabx}

\newtheorem{theorem}{Theorem}[section]
\newtheorem{lemma}[theorem]{Lemma}
\newtheorem{proposition}[theorem]{Proposition}
\newtheorem{corollary}[theorem]{Corollary}

\newtheorem{definition}[theorem]{Definition}

\newtheorem{remark}[theorem]{Remark}

\newcommand{\V}{\mathcal{V}}
\newcommand{\eq}{\approx}

\title{Modal expansions of ririgs}
\author[1,2]{Agustín L. Nagy}
\author[1]{William J. Zuluaga Botero}
\affil[1]{Departamento de Matem\'atica, Facultad de Ciencias Exactas (UNCPBA), Pinto 399, Tandil, Argentina}
\affil[2]{CONICET, Argentina}
\date{}

\begin{document}

\maketitle
\begin{abstract}
In this paper we introduce the variety of I-modal ririgs. We characterize the congruence lattice of its members by means of I-filters and we provide a description on I-filter generation. We also provide an axiomatic presentation for the variety generated by chains of the subvariety of contractive I-modal ririgs. Finally, we introduce a Hilbert-style calculus of a logic with I-modal ririgs as an equivalent algebraic semantics and we prove that such a logic has the parametrized local deduction-detachment theorem. 
\end{abstract}

\section{Introduction}

Classical modal logic can be considered as the study of the deductive behavior of the expressions ``it is necessary that" and ``it is possible that". It is motivated by the limitations of Propositional Classical logic (CP) in deciding whether certain propositions are true or false. In 1932, Lewis and Langford \cite{LL1932} expand the language of CP with a unary connective $\Box$ and add the distribution axiom 
\[\Box (\varphi \rightarrow \psi) \rightarrow \Box \varphi \rightarrow \Box \psi\]
together with the rules modus ponens and $\Box$-necessitation (if $\varphi$ is a theorem so $\Box \varphi$). This system later received the name of $\mathbf{K}$ system in honor of Saul Kripke. Thereafter, stronger systems appeared. This is the case of the $\mathbf{S4}$ system, which has been extensively studied along the literature (see \cite{BRV2001} for a general survey) and it has being successfully applied to artificial intelligence \cite{MvH1995}. It is the axiomatic extension of the $\mathbf{K}$ system given by adding the axioms $\Box \varphi \rightarrow \varphi$ and $\Box \varphi \rightarrow \Box\Box\varphi$. In the late 1950s Prior introduces tense logic \cite{P1957}. It was a logic that wanted to deal with the philosophical implications of free will and predestination. Its language is the language of CP expanded by two unary connectives $G$ and $H$ which satisfy the distribution axiom. Its rules are modus ponens and $G$ and $H$-necessitation. 
This procedure of adding unary connectives satisfying the distribution axiom to some ``base logic", moved the interest in study non-classical modal logics. This is the case of intuitionistic modal logics $\mathbf{K}_{\Box}$ and the $\mathbf{S4K}_{\Box}$ (see \cite{WZ1999} and all the references therein), intuitionistic temporal logic \cite{E1986} and more recently, the logic $\mathbf{\L}(I)$, which is an expansion of the \L ukasiewicz logic $\mathbf{\L}$ by a set $I$ of unary connectives satisfying the distribution axiom and the rules $\Box$-necessitation for all $\Box\in I$ and modus ponens. 

It is well known that all the logics mentioned above, have an equivalent algebraic semantics. This algebraic correlate has motivated by itself the introduction and study of a huge amount of algebraic structures with a set of unary operators whose behavior can be interpreted as the algebraic behavior of certain modal operators. This is the case of modal algebras, tense algebras, S4-algebras, modal Heyting algebras and $\mathbf{MV}(I)$-algebras, to mention some. Such an approach has revealed deep advantages at the moment of studying diverse classical and non-classical modal logics.  Our motivation has to do with exploiting this approach. That is to say, in this paper we are intended to study, from an algebraic perspective, the consequences of the procedure of adding a set of modal operators to the algebraic class of residuated integral rigs. We have chosen this structures due to their relation with residuated lattices and therefore, with substructural logics. 
\\

The paper is organized as follows. In Section \ref{Preliminaries} we present the notions and results about residuated integral rigs and deductive systems which we will use throughout this paper. In Section \ref{I-modal ririgs} we introduce the structures we study along this work, namely I-modal ririgs. We characterize the congruences of the members of this class of algebras by means of I-filters and furthermore, we give a suitable description on I-filter generation. Afterwards, we dedicate some efforts to the I-modal ririgs with I finite. We show that such a condition over the set I, makes that the I-filter generation can be stated in terms of a single modal operator. In Section \ref{The variety generated by chains} we describe the variety generated by chains of the subvariety of contractive I-modal ririgs in terms of a quite simple set of equations. The content of Section \ref{Compatible functions} will be dedicated to the study of compatible operations and equationally defined compatible operations of I-modal ririgs. Finally, In Section \ref{logic R(I)} we present a Hilbert-style calculus for a modal logic \( \mathcal{S}_{\mathcal{H}} \) and we will prove that such a logic has as an equivalent variety semantics the variety of I-modal ririgs. The paper concludes by showing that  \( \mathcal{S}_{\mathcal{H}} \) has the local deduction-detachment theorem.

We assume the reader is familiar with the basics of universal algebra \cite{BS1981}, residuated lattices \cite{GJKO2007} and abstract algebraic logic.   

\section{Preliminaries}\label{Preliminaries}

\subsection{Integral residuated rigs}

A \emph{rig} is an algebra $\mathbf{A}=(A, +, \cdot,0, 1)$ of type $(2,2,0,0)$ such that the structures $(A, \cdot, 1)$ and $(A, +, 0)$ are commutative monoids such that ``product distributes over addition" in the sense that $x \cdot 0 = 0$ and $x \cdot (y + z) = (x \cdot y) + (x \cdot z)$ for every $x, y, z \in A$. One may think such structures as ``(commutative) rings (with unit) without negatives". A rig is said to be \emph{integral} if the equation $1+x=1$ holds, for every $x\in A$. Observe that the latter implies that $1+1=1$ so $+$ becomes idempotent. This makes of $(A,+,0,1)$ a bounded join semilattice. In what follows, and in order to avoid confusion, we write $\vee$ to denote $+$, as usual. We say that an integral rig is \emph{residuated} if for all $a, b, c \in A$:

\begin{displaymath}
\begin{array}{ccc}
a\cdot b \leq c & \Leftrightarrow & a\leq b\rightarrow c.
\end{array}
\end{displaymath}

We stress that the latter is equivalent to say that for every $b\in A$, the map $b\cdot (-):A\rightarrow A$ is a left adjoint of the map $b\rightarrow (-):A\rightarrow A$. Observe that from general reasons, this fact implies that the class of residuated integral rigs is a variety. By \emph{ririg} we mean integral residuated rig. Unless any clarification is needed, in the rest of this paper we write $xy$ instead of $x\cdot y$. Moreover, we write $\prod_{j=1}^{k}x_j$ for $x_1\cdot \ldots \cdot x_k$. 
\\

The following result, whose proof is analogue to Lema 7.1.6 of \cite{Ch2007} provides some useful properties for ririgs.

\begin{lemma}\label{props riRigs}
Let $\mathbf{A}$ be a ririg and let $a, b, c\in A$. Then, the following hold:
\begin{itemize}
\item[1)] $a\rightarrow 1 = 1$,
\item[2)] $1\rightarrow a = a$,
\item[3)] $a\rightarrow a = 1$,
\item[4)] $a(a\rightarrow b)\leq b$,
\item[5)] $a\rightarrow b\leq ac\rightarrow bc$,
\item[6)] $a\leq (a\rightarrow b)\rightarrow b$,
\item[7)] If $a\leq b$ then $c\rightarrow a\leq c\rightarrow b$ and $b\rightarrow c\leq a\rightarrow c$,
\item[8)] $a\rightarrow (b\rightarrow c)=ab\rightarrow c = b\rightarrow (a\rightarrow c)$,
\item[9)] $a\leq b\rightarrow c$ if and only if $b\leq a\rightarrow c$,
\item[10)] $a\leq b\rightarrow a$,
\item[11)] $a\rightarrow b\leq (c\rightarrow a)\rightarrow (c\rightarrow b)$,
\item[12)] $a\rightarrow b\leq (b\rightarrow c)\rightarrow (a\rightarrow c)$,
\item[13)] $a\rightarrow b = ((a\rightarrow b)\rightarrow b )\rightarrow b)$.
\item[14)] $a\leq b$ if and only if $a\rightarrow b=1$.
\end{itemize}
\end{lemma}

\subsection{Deductive systems}

By a \emph{propositional language} $\mathcal{L}$ we mean a set of propositional connectives, each of given
finite arity. In an algebraic context the elements of $\mathcal{L}$ are usually known as fundamental operations. If $\mathcal{L}$ is a propositional language we write \( \mathbf{Fm_{\mathcal{L}}} \) for the absolutely free algebra of formulas in the language \( \mathcal{L}\). If there is no risk of confusion we write \( \mathbf{Fm} \) instead \( \mathbf{Fm_{\mathcal{L}}} \). Besides, we write \( \mathrm{Fm} \) to indicate the universe of the algebra \( \mathbf{Fm} \). An equation (or identity) is a pair of formulas \( \left( \varphi, \psi \right) \). We use \( \varphi \approx \psi \) to write the equation \( \left( \varphi,\psi \right) \) and \( \mathrm{Eq} \) to denote the set of equations. Let \( \mathbf{A} \) be an algebra of type \( \mathcal{L} \) and \( v \colon \mathbf{Fm} \to \mathbf{A} \) a homomorphism. We say that the identity \( \varphi \approx \psi \) is valid on \( \mathbf{A} \), under the homomorphism \( v \), provided that \( v^{\mathbf{A}}\left( \varphi \right) = v^{\mathbf{A}}\left( \psi \right) \). We denote this by \( \left(A,v\right) \vDash \varphi \approx \psi \). We say that the equation \( \varphi \approx \psi \) is valid on \( \mathbf{A} \)  if for each homomorphism \( v \colon \mathbf{Fm} \to \mathbf{A} \) it holds \( \left(A,v \right) \vDash \varphi \approx \psi \). We write \( \mathbf{A} \vDash \varphi \approx \psi \) to say that the identity \( \left( \varphi,\psi \right) \) is valid on \( \mathbf{A} \). If \( \mathcal{K} \) is a class of algebras of type \( \mathcal{L}\) we say that the class \( \mathcal{K} \) satisfies the identity \( \left( \varphi,\psi \right) \) if for each member \( \mathbf{A} \in \mathcal{K} \) it holds \( \mathbf{A} \vDash \varphi \approx \psi \). We use \( \mathcal{K} \vDash \varphi \approx \psi \) to say that the identity \( \varphi \approx \psi \) is valid on \( \mathcal{K} \). If \( \Theta \) is a set of equations then we say that the set \( \Theta \) is valid on the class \( \mathcal{K} \) provided that \( \mathcal{K} \vDash \alpha \approx \beta \) for each \( \alpha \approx \beta \in \Theta \). We write \( \mathcal{K} \vDash \Theta \) to denote this fact.

A sequent of \( \mathbf{Fm} \) is a pair \( \left( \Gamma,\varphi \right) \) where \( \Gamma \) is a possibly empty set of formulas and \( \varphi \) is a formula. We write \( \Gamma \rhd \varphi \) to denote the sequent \( \left( \Gamma,\varphi \right) \) and we denote by \( \mathrm{Seq} \) the set of all sequents of \( \mathbf{Fm} \). A Hilbert rule \( \langle \Gamma,\varphi \rangle \) is the closure by endomorphisms on \( \mathbf{Fm} \) of the sequent \( \Gamma \rhd \varphi \). I.e.:
\[ \langle \Gamma,\varphi \rangle = \{ \sigma[\Gamma] \rhd \sigma[\varphi] \colon \sigma \colon \mathbf{Fm} \to \mathbf{Fm} \text{ is a endomorphism} \}. \]

A relation \( \vdash \,\ \subseteq \mathcal{P}\left( \mathbf{Fm} \right) \times \mathbf{Fm} \) is said to be a \emph{consequence relation} or entailment relation, provided that: 
\begin{itemize}
    \item[(1)] If \( \varphi \in \Gamma \), then \( \Gamma \vdash \varphi \).
    \item[(2)] If \( \Gamma \vdash \varphi \) and \( \Gamma \subseteq \Delta \). then \( \Delta \vdash \gamma \).
    \item[(3)] If \( \Gamma \vdash \varphi \) and \( \Delta \vdash \gamma \) for each \( \gamma \in \Gamma \). then \( \Delta \vdash \varphi \).
    \item[(4)] If \( \Gamma \vdash \varphi \) and \( \sigma \colon \mathbf{Fm} \to \mathbf{Fm} \) is an endomorphism, then \( \sigma[\Gamma] \vdash \sigma(\varphi) \).
\end{itemize}

We say that the pair \( S = \langle \mathcal{L},\vdash_{S} \rangle \) is a \emph{deductive system} (or a \emph{logic}) whenever \( \vdash_{S} \) is a consequence relation on the universe of the algebra \( \mathbf{Fm} \) of signature \( \mathcal{L} \). If there is no risk of confusion we write \( \langle \mathbf{Fm},\vdash_{S} \rangle \) instead \( \langle \mathcal{L},\vdash_{S} \rangle \) to denote the deductive system \( S \) and the signature will be clear from the context. We say that the sequent \( \Gamma \rhd \varphi \) is a sequent of \( S \) provided \( \Gamma \vdash_{S} \varphi \). Let \( \langle \Gamma,\varphi \rangle \) be a Hilbert rule. We say that \( \langle \Gamma,\varphi \rangle \) is a rule of the deductive system \( S \) whenever \( \Gamma \vdash_{S} \varphi \).

If \( \mathcal{K} \) is a class of algebras of a given type, then the \emph{equational consequence relation} associated to \( \mathcal{K} \) is defined as follows: We say that the equation \( \varphi \approx \psi \) follows from the set of equations \( \Theta \) on the class of algebras \( \mathcal{K} \) whenever \( \mathcal{K} \vDash \Theta \) implies \( \mathcal{K} \vDash \varphi \approx \psi \). I.e.:
\[ \Theta \vDash_{\mathcal{K}} \varphi \approx \psi \text{ if only if } \mathcal{K} \vDash \Theta \text{ implies } \mathcal{K} \vDash \varphi \approx \psi.
\]

A deductive system \( S = \langle \mathbf{Fm},\vdash_{S} \rangle \) is said to be  \emph{Block-Pigozzi algebraizable} \cite{BP} if there is a class of algebras \( \mathcal{K} \) on the signature of \( \mathbf{Fm} \) and structural transformers \( \tau \colon \mathrm{Fm} \to \mathcal{P}\left(\mathrm{Eq} \right) \), \( \rho \colon \mathrm{Eq} \to \mathcal{P}\left(\mathrm{Fm}\right) \) such that for each \( \Gamma, \varphi \subset \mathrm{Fm} \) and \( \Theta, \alpha \approx \beta \in \mathrm{Eq} \) the following hold:  

\begin{itemize}
\item[(1)] \( \Gamma \vdash_{S} \varphi \) if and only if \( \tau[\Gamma] \vDash_{\mathcal{K}} \tau(\varphi) \),
\item[(2)] \( \Theta \vDash_{\mathcal{K}} \alpha \approx \beta \) if and only if \( \rho \left( \Theta \right) \vdash_{S} \rho\left( \alpha \approx \beta \right) \),
\item[(3)] \( \varphi \dashv \vdash_{S} \rho(\tau(\varphi)) \),
\item[(4)] \( \alpha \approx \beta \Dashv \vDash_{\mathcal{K}} \tau \left( \rho \left( \alpha \approx \beta \right) \right) \) .
\end{itemize}

\begin{remark} \label{remark algebraizacion} Notice that a deductive system \( S \) is algebraizable with equivalent algebraic semantic \( \mathcal{K} \) and structural transformers \( \tau \colon \mathrm{Fm} \to \mathcal{P}\left( \mathrm{Eq} \right) \) and \( \rho \colon \mathrm{Eq} \to \mathcal{P}\left( \mathrm{Fm} \right) \) if and only if the following hold:
\begin{itemize}
\item[(1)] \( \Gamma \vdash_{S} \varphi \) if and only if \( \tau[\Gamma] \vDash_{\mathcal{K}} \tau(\varphi) \),
\item[(2)] \( \varphi \approx \psi \Dashv \vDash_{\mathcal{K}} \tau\left( \rho(\varphi \approx \psi) \right) \).
\end{itemize}
\end{remark}

\section{I-modal ririgs}\label{I-modal ririgs}

Let $\mathbf{A}$ be a ririg and let $m$ be a unary operator on $A$. We say that $m$ is a \emph{modal operator} on $\mathbf{A}$ if for every $x,y\in A$, the following identities hold:
\begin{displaymath}
\begin{array}{ccc}
m(x\rightarrow y)\leq m(x)\rightarrow m(y) & \text{and} & m(1)=1.
\end{array}
\end{displaymath}

\begin{remark}\label{important prop}
Let $\mathbf{A}$ be a ririg and let $m$ be a unary operator on $A$. Notice that one easily can check that from residuation it is the case that $m(x\rightarrow y)\leq m(x)\rightarrow m(y)$ holds if and only if $m(x)m(y)\leq m(xy)$ holds for every $x,y\in A$. Moreover, from Lemma \ref{props riRigs} (14), it is also clear that $m$ is a monotone operator on $A$.
\end{remark}

Now we present the structures we will study throughout this paper.

\begin{definition}\label{MTL(I)-algebra}
Let I be a set of unary function symbols. We say that an algebra $\mathbf{A}=(A,\vee, \cdot, \rightarrow, 0, 1, \{m\}_{m\in I})$ is an I-modal ririg provided that:
\begin{itemize}
    \item[1.] $(A, \vee, \cdot, \rightarrow, 0, 1)$ is a ririg.
    \item[2.] For every $m\in I$, $m$ is a modal operator on $A$.
\end{itemize}
\end{definition} 

It is immediate from Definition \ref{MTL(I)-algebra} that for every set of unary function symbols $I$, the class $\mathcal{R}(I)$ of I-modal ririgs is a variety.

\begin{definition}\label{MTL(I)-filter}
Let $\mathbf{A}$ be an I-modal ririg. We say that a non-empty subset $F$ of $A$ is an I-filter provided:
\begin{itemize}
\item[1.] $F$ is an up-set,
\item[2.] $F$ is closed under $\cdot$,
\item[3.] $F$ is closed under each $m \in I$.
\end{itemize}
\end{definition}

Let $\mathbf{A}$ be an I-modal ririg. In what follows, we will consider the term operation $\ast$, defined as $x\ast y = (x \rightarrow y)(y \rightarrow x)$. We write $\mathsf{Fi}(\mathbf{A})$ for the poset of I-filters of $\mathbf{A}$ ordered by inclusion and $\mathsf{Con}(\mathbf{A})$ for the congruence lattice of $\mathbf{A}$. Moreover, if $x,y\in A$, we denote by ${\mathsf{Cg}^{\mathbf{A}}(x,y)}$ the smallest congruence of $\mathbf{A}$ containing the pair $(x,y)$.
\\

The following result provides a characterization of the congruences of the members of $\mathcal{R}(I)$ by means of I-filters. Its proof is analogous to the one of Lemma 4 in \cite{FZ2021}, so we leave the details to the reader.

\begin{lemma}\label{Congruence I-Filters}
Let $\mathbf{A}$ be an I-modal ririg, $F\in \mathsf{Fi}(\mathbf{A})$, and $\theta\in \mathsf{Con}(\mathbf{A})$. Then the following hold:
\begin{enumerate}
\item $F_{\theta}=1/\theta$ is an I-filter of $\mathbf{A}$.
\item The set $\theta_{F}=\{(x,y)\in A^{2}\colon x\ast y\in F \} = \{(x,y)\in A^2\colon x\leftrightarrow y\in F\}$ is a congruence on $\mathbf{A}$.
\item The maps $F\mapsto \theta_{F}$, $\theta \mapsto F_{\theta}$ define a poset isomorphism between $\mathsf{Con}(\mathbf{A})$ and $\mathsf{Fi}(\mathbf{A})$. Consequently, $\mathsf{Fi}(\mathbf{A})$ is a lattice and these poset isomorphisms are lattice isomorphisms.
\end{enumerate}
\end{lemma}

Now, we will focus on the description of the filters generated by sets in $\mathcal{R}(I)$. To this end, we need to recall some notions first. If $I$ denotes a set of unary connective symbols, an \emph{$I$-block} is a word in the alphabet $I$.  We denote the set of $I$-blocks by $\mathcal{B}_{I}$ and the empty word in the alphabet $I$ by $\varepsilon$. Let $\mathbf{A} \in \mathcal{R}(I)$ and $m\in I$. We write $m^{\mathbf{A}}$ for the \emph{interpretation of $m$} in $\mathbf{A}$. We define the \emph{interpretation of $I$-blocks} in $\mathbf{A}$ recursively, as follows:
\begin{itemize}
\item[1.] If $m\in I$, then $m^{\mathbf{A}}$ is a function $m^{\mathbf{A}}:A\rightarrow A$. 
\item[2.] If $M=\varepsilon$, then $\varepsilon^{\mathbf{A}}=id_{A}$.
\item[3.] If $m\in I$, $N\in \mathcal{B}_{I}$ and $M$ is the word $mN$, then $M^{\mathbf{A}}$ is defined as the composition of $m^{\mathbf{A}}$ with $N^{\mathbf{A}}$. I.e. $M^{\mathbf{A}}=m^{\mathbf{A}}N^{\mathbf{A}}$.
\end{itemize}

Notice that for every $\mathbf{A}\in \mathcal{R}(I)$, $m^{\mathbf{A}}$ is a modal operator on $A$ and also that $id_{A}$, the identity on $A$, is a modal operator on $A$. This makes the proof of the following fact a straightforward consequence of the definition of interpretation of $I$-blocks on I-modal ririgs. 

\begin{proposition}
Let $\mathbf{A} \in \mathcal{R}(I)$. Then, $M^{\mathbf{A}}$ is a modal operator on $\mathbf{A}$ for every $M\in \mathcal{B}_{I}$.
\end{proposition}

If there is no place to confusion, in the following we will omit the superscripts on the interpretation of $I$-blocks on I-modal ririgs. 
\\

We recall that if $(P, \leq)$ is a partially ordered set and $X \subseteq P$, then the smallest up-set containing $X$ is the set 
\[\uparrow X = \{y \in P \mid x \leq y\; \text{for some}\; x \in X\}.\]
The following result characterizes congruence generation in $\mathcal{R}(I)$ by means of the least up-set containing finite products of finitely many I-blocks evaluated in finitely many elements of a given set. Its proof is analogue to the one of Lemma 5 in \cite{FZ2021}, so we skip the details. 

\begin{lemma}\label{generated filters R(I)}
Let $\mathbf{A}$ be an I-modal ririg and let $X \subseteq A$. Then, the set 
\[\mathsf{Fg}^{\mathbf{A}}(X)=\uparrow \{M_1(x_1)\cdot\ldots\cdot M_n(x_n)\colon x_1,\ldots,x_n \in X\;\text{and}\; M_1,...,M_n\in\mathcal{B}_{I}\}. \]
is the least I-filter of $\mathbf{A}$ containing $X$.
\end{lemma}

If $X=\{x_{1},...,x_{n}\}$, we will write $\mathsf{Fg}^{\mathbf{A}}(X)$ simply as $\mathsf{Fg}^{\mathbf{A}}(x_{1},...,x_{n})$. Next, we present a characterization of the $I$-filters determined by the smallest congruences of $\mathbf{A}$ containing the pair $(x,y)$, for a given a pair of elements $x,y$ in $A$ and the set $X=\{(1,y)\colon a\in Y\}$, respectively. This result will be particularly useful in Section \ref{Compatible functions}.

\begin{lemma}\label{cor:Generated filter intersection}
Let $\mathbf{A}\in \mathcal{R}(I)$, let $x,y\in A$, let $Y\subseteq A$, and consider $X=\{(1,y)\colon a\in Y\}$. Then:
\begin{enumerate}
\item $F_{\mathsf{Cg}^{\mathbf{A}}(x,y)}=\mathsf{Fg}^{\mathbf{A}}(x\ast y)=\mathsf{Fg}^{\mathbf{A}}(x\leftrightarrow y)$.
\item $F_{\mathsf{Cg}^{\mathbf{A}}(X)}=\mathsf{Fg}^{\mathbf{A}}(Y)$.
\end{enumerate}
\end{lemma}
\begin{proof}
1. Note that $\mathsf{Cg}^{\mathbf{A}}(x,y)=\bigcap \{\theta\in \mathsf{Con}(\mathbf{A})\colon (x,y)\in\theta\}$, and observe that for each $\theta\in\mathsf{Con}(\mathbf{A})$ we have $(x,y)\in\theta$ if and only if $x\ast y \in F_{\theta}$. Hence from the isomorphism given by Lemma \ref{Congruence I-Filters}(3) we obtain:
$$F_{\mathsf{Cg}^{\mathbf{A}}(x,y)}=\bigcap \{F\in \mathsf{Fi}(\mathbf{A})\colon x\ast y\in F\}=\mathsf{Fg}^{\mathbf{A}}(x\ast y)=\mathsf{Fg}^{\mathbf{A}}(x\leftrightarrow y).$$
This proves 1.

2. Since $\mathsf{Cg}^{\mathbf{A}}(X)=\bigvee_{y\in Y}\mathsf{Cg}^{\mathbf{A}}(1,y)$, Lemma \ref{Congruence I-Filters}(3) and item 1 imply
\[F_{\mathsf{Cg}^{\mathbf{A}}(X)}=\bigvee_{y\in Y} F_{\mathsf{Cg}^{\mathbf{A}}(1,y)}=\bigvee_{y\in Y} \mathsf{Fg}^{\mathbf{A}}(y)=\mathsf{Fg}^{\mathbf{A}}(\bigcup_{y\in Y} \{y\})=\mathsf{Fg}^{\mathbf{A}}(Y).\]
This proves 2.
\end{proof}

\subsection{The finite case}\label{The finite case}

In this section we study the variety of finite I-modal ririgs. We will show that the results of Section \ref{I-modal ririgs} can be expressed by means of a single unary operation which is constructed from the modal operators belonging to I. To this end, let $I=\{m_{1},...,m_{k}\}$ be a finite set of unary function symbols. We write $\mathcal{R}(I_{\omega})$ for the variety $\mathcal{R}(I)$ when $I$ is finite. We define the unary operation $\lambda$ as follows:
\[\lambda(x)=x\cdot \prod^{k}_{j=1}m_{j}(x). \]
Moreover, we consider $\lambda^{0}(x)=x$ and $\lambda^{l+1}(x)=\lambda(\lambda^{l}(x))$, for every $l\in \mathbb{N}$.
\begin{lemma}\label{propiedades de lambda}
Let $\mathbf{A}$ be an I-modal ririg. Then for every $x,y\in A$ the following hold:
\begin{enumerate}
\item $\lambda(x)\leq x$.
\item $\lambda(0)=0$ and $\lambda(1)=1$.
\item The operator $\lambda$ is a modal operator.
\item $\lambda^{l+1}(x)\leq \lambda^{l}(x)$, for every $l\in \mathbb{N}$.
\end{enumerate}
\begin{proof}
Notice that $(1)$ follows from the definition of $\lambda$. Observe that (2) follows from (1) and the fact that $m_{j}(1)=1$ for every $1\leq j\leq k$. For (3), it is clear from (2) that $\lambda(1)=1$, thus only remains to prove $\lambda(x\rightarrow y)\leq \lambda(x)\rightarrow \lambda(y)$. To do so, we will prove that $\lambda(x)\lambda(y)\leq \lambda(xy)$. Since $m_{j}(x\rightarrow y)\leq m_{j}(x)\rightarrow m_{j}(y)$ for every $1\leq j\leq k$, then by Remark \ref{important prop} we have $m_{j}(x)m_{j}(y)\leq m_{j}(xy)$. Thus, due $\cdot$ is commutative and order preserving we obtain 
\[\prod_{j=1}^{k}m_{j}(x)\cdot \prod_{j=1}^{k}m_{j}(y)\leq \prod_{j=1}^{k}m_{j}(xy), \]
so, multiplying by $xy$ at both sides of the latter inequality and by applying the same argument we employed before, we may conclude $\lambda(x)\lambda(y)\leq \lambda(xy)$, as claimed. Hence, by Remark \ref{important prop}, $\lambda(x\rightarrow y)\leq \lambda(x)\rightarrow \lambda(y)$, as desired. For (4) we apply induction on $l$. It is clear by (1) that the statement is true for $l=0$, so assume as inductive hypothesis $\lambda^{l+1}(x)\leq \lambda^{l}(x)$. By (3), $\lambda$ is monotone, so we obtain $\lambda^{l+2}(x)=\lambda(\lambda^{l+1}(x))\leq \lambda(\lambda^{l}(x))=\lambda^{l+1}(x)$, as required. This concludes the proof.
\end{proof}
\end{lemma}

\begin{proposition}\label{filters finite case}
Let $\mathbf{A}$ be an I-modal ririg and let $F\subseteq A$. Then $F$ is an I-filter if and only if $F$ is a filter of ririgs closed by $\lambda$.
\end{proposition}
\begin{proof}
On the one hand, let $F$ be an I-filter of $\mathbf{A}$. Then $m_{j}(x)\in $, for every $x\in F$ and $1\leq j\leq k$. Since $F$ is closed by products, $\lambda(x)\in F$. On the other hand, let $F$ be a filter of ririgs closed by $\lambda$. Then, from definition of $\lambda$, we obtain $\lambda(x)\leq m_{j}(x)$, for every $x\in F$. Thus, since $F$ is increasing, the result follows.
\end{proof}

\begin{lemma}\label{generated filter finite case}
Let $\mathbf{A}$ be an I-modal ririg and $X\subseteq A$. Then 
\[\mathsf{Fg}^{\mathbf{A}}(X)=\uparrow \{\prod_{i=1}^{n}\lambda^{l}(x_{i}) \colon x_{1},...,x_{n}\in X\; \text{ and }\; l\in \mathbb{N}\}. \]
\end{lemma}
\begin{proof}
Let $Z=\uparrow \{\prod_{i=1}^{n}\lambda^{l}(x_{i}) \colon x_{1},...,x_{n}\in X\; \text{ and }\; l\in \mathbb{N}\}$. We start by proving that $Z$ is an I-filter. To do so, by Proposition \ref{filters finite case} we need to prove that $Z$ is a filter of ririgs closed under $\lambda$. Let $y\in Z$, then there exist $l\in \mathbb{N}$ and $x_{1},...,x_{n}\in X$ such that $\prod_{i=1}^{n}\lambda^{l}(x_{i})\leq y$. Then, from Proposition \ref{propiedades de lambda} (3),  \[\prod_{i=1}^{n}\lambda^{l+1}(x_{i})\leq \lambda(\prod_{i=1}^{m}\lambda^{l}(x_{i}))\leq \lambda(y),\]
so $\lambda(y)\in Z$. Now we prove that $Z$ is closed under $\cdot$. Let $y_{1},y_{2}\in Z$. Then there exist $l,q\in \mathbb{N}$ and $a_{1},...,a_{n},b_{1},...,b_{p}\in X$ such that $\prod_{i=1}^{n}\lambda^{l}(a_{i})\leq y_{1}$ and $\prod_{i=1}^{p}\lambda^{q}(b_{i})\leq y_{2}$. Therefore from Proposition \ref{propiedades de lambda} (4) we get $\lambda^{l+q}(a_{i})\leq \lambda^{l}(a_{i})$, for every $1\leq i\leq n$ and $\lambda^{l+q}(b_{j})\leq \lambda^{l}(b_{j})$, for every $1\leq j\leq p$. So
\[\prod_{i=1}^{n}\lambda^{l+q}(a_{i})\cdot\prod_{i=1}^{p}\lambda^{l+q}(b_{i})\leq  y_{1}\cdot y_{2},\]
thus $y_{1}\cdot y_{2}\in \mathsf{Fg}^{\mathbf{A}}(X)$. Finally, $Z$ is non-empty since $1\in Z$ and it is clear that $Z$ is an up-set. Hence, $Z$ is a filter of ririgs closed by $\lambda$, as claimed.

For the last part, notice that $X\subseteq Z$ by Proposition \ref{propiedades de lambda} (1). Now, let $F$ be an I-filter such that $X\subseteq F$ and take $y\in Z$. Then there exist $l\in \mathbb{N}$ and $x_{1},...,x_{n}\in X$ such that $\prod_{i=1}^{n}\lambda^{l}(x_{i})\leq y$. Since $x_{i}\in F$ for every $1\leq i\leq n$, then by Proposition \ref{filters finite case} $\lambda(x_{i})\in F$ and therefore $\lambda^{l}(x_{i})\in F$. Since $F$ is closed by $\cdot$, $\prod_{i=1}^{n}\lambda^{l}(x_{i})\in F$ and because $F$ is increasing, $y\in F$ as claimed.
\end{proof}

We conclude this section by characterizing the simple and subdirectly irreducible members of \( \mathcal{R}(I) \). The following results provide a generalization of the results given in \cite{C,H}. 

\begin{theorem}
Let \( \mathbf{A} \in \mathcal{R}(I) \) non trivial. Then \( \mathbf{A} \) is simple if and only if for each \( a \neq 1 \) there is  \( M \in \mathcal{B}_{I} \) such that \( M(a) = 0 \). Moreover, if $I$ is finite, then  $\mathbf{A}$ is simple if and only if, for each \( a \neq 1 \) there exists $l\in \mathbb{N}$ such that $\lambda^{l}(a)=0$.
\end{theorem}
\begin{proof}
Suppose that \( \mathbf{A} \) is simple but no trivial algebra and let \( a \in \mathbf{A} \). Then, by Lemma \ref{Congruence I-Filters}, we obtain \( \mathsf{Fg}^{\mathbf{A}}\left( \{a\} \right)= A \). Thus by Lemma \ref{generated filters R(I)}, there is \( M \in \mathcal{B}_{I} \) with \( M(a) = 0 \). On the other hand, suppose that for every \( a \in \mathbf{A} \) with \( a \neq 1 \) there is \( M \in \mathcal{B}_{I} \) such that \( M(a) = 0 \). We shall prove that \( \mathsf{Fi}(\mathbf{A}) = \{ \{1\}, \mathbf{A} \} \). Let \( F \in \mathsf{Fi}(\mathbf{A}) \) and \( a \in F \). By assumption we have that \( 0 \in \mathsf{Fg}\left( \{a\} \right) \subseteq F \) so \( 0 \in F \) and \( F = \mathbf{A} \).  The moreover part follows from Lemma \ref{generated filter finite case}.
\end{proof}

\begin{theorem} 
Let \( \mathbf{A} \in \mathcal{R}(I) \) non trivial algebra. Then \( \mathbf{A} \) is subdirectly irreducible if only if there is \( b \in \mathbf{A} \), \( b \neq 1 \) such that for each \( a \in \mathbf{A} \), \( a \neq 1 \), there is \( M \in \mathcal{B}_{I} \) with \( M(a) \leq b \). Moreover, if $I$ is finite, then  $\mathbf{A}$ is subdirectly irreducible if only if there exists $l\in \mathbb{N}$ such that $\lambda^{l}(a)\leq b$.
\end{theorem}
\begin{proof}
Suppose that \( \mathbf{A} \) is a subdirectly irreducible algebra then the lattice \( \mathsf{Fi}\left( \mathbf{A} \right) \) has a monolith \( F \neq \{1\} \). Hence there is \( b \in F \) with \( b \neq 1 \). Note that for each \( a \in \mathbf{A} \) we have that \( b \in \mathsf{Fg}(\mathbf{A}) \). Thus there is \( M \in \mathcal{B}_{I} \) with \( M(a) \leq b \).

Conversely suppose that there is \( b \in \mathbf{A} \) such that for each \( a \in \mathbf{A} \) with \( a \neq 1 \) there is \( M \in \mathcal{B}_{I} \) with \( M(a) \leq b \). Note that \( \mathsf{F_{g}}\left( \{b\} \right) \) is a monolith in \( \mathsf{Fi}(\mathbf{A}) \). By Lemma \ref{Congruence I-Filters} we have that \( \mathbf{A} \) is subdirectly irreducible. The moreover part follows from Lemma \ref{generated filter finite case}.
\end{proof}

\section{The variety generated by chains} \label{The variety generated by chains}

In this section we characterize the variety generated by chains of the subvariety of \( \mathcal{R}(I) \) whose members are \emph{contractive} I-modal ririgs. I.e. we say that $\mathbf{A}\in \mathcal{R}(I)$ is contractive, if for every $x\in A$ and $m\in I$, the equation $m(x)\leq x$ holds. We denote such a subvariety by $\mathcal{R}_{c}(I)$. Along this section we write \( \mathcal{C} \) for the subclass of \( \mathcal{R}_{c}(I) \) whose members are chains and  \( \mathcal{V}\left( \mathcal{C} \right) \) for the variety generated by it. We stress that the results of this section are an specialization of the results of  \cite{HRT}.

We start by considering the following identities in the signature of I-modal ririgs, where $m$ ranges over I:
\begin{itemize}
\item[(P)] $(a \to b) \vee (b \to a) = 1,$
\item[(Cm)] $m(a \vee b) \leq  m(a) \vee m(b)$.
\end{itemize}

We write $\mathcal{R_{C}}(I)$ for the subvariety of $\mathcal{R}_{c}(I)$ whose members satisfy the identities (P) and (Cm), for every $m\in I$. Notice that \( \mathcal{C} \) is a proper subclass of \( \mathcal{R_{C}}(I) \),  so \( \mathcal{V}(\mathcal{C}) \) is a proper subvariety of \( \mathcal{R_{C}}(I). \) Along this section we are intended to prove that $\mathcal{R_{C}}(I)$ and $\mathcal{V}(\mathcal{C})$ coincide. To do so, we need to first prove some properties of the members of $\mathcal{R_{C}}(I)$.

\begin{lemma}\label{lemma M0}
Let $\mathbf{A} \in \mathcal{R_{C}}(I)$ and $a,b \in A$. Then, for every \( m, n \in I \), 
\[ mn(a \vee b) \leq  a \vee n(b). \] 
\end{lemma}
\begin{proof}
Immediate from the contractivity of $m$ and $n$.
\end{proof}

\begin{lemma}\label{lemma M1}
Let $\mathbf{A} \in \mathcal{R_{C}}(I)$ and $a,b \in A$. Then, for every \( m, n \in I \), there exists $Q\in \mathcal{B}_{I}$ such that:
\[ Q(a \vee b) \leq  m(a) \vee n(b). \] 
\end{lemma}
\begin{proof}
Let $m,n \in I$ and $a,b\in I$. Then from Lemma \ref{lemma M0}, $ mn(a \vee b) \leq  a \vee n(b)$. Now take $Q=mmn$. Then, from (Cm), (Cn) and the contractivity of $m$ and $n$ we get 
\[Q(a\vee b)\leq m(a)\vee mn(b)\leq m(a)\vee n(b),\] 
as required.
\end{proof}

\begin{lemma}
Let \( \mathbf{A} \in \mathcal{R_{C}}(I) \) and $a,b \in A$. Then, for every $M\in \mathcal{B}_{I}$ and $n \in I$, there exists $N\in \mathcal{B}_{I}$ such that 
\[Q(x\vee y)\leq M(x) \vee n(y). \]
\end{lemma}
\begin{proof}
Let $M$ be an $I$-block and let $n \in I$ fixed. In order to prove our claim, we proceed by induction on $\mathcal{B}_{I}$. If $M=m\in I$, then from Lemma \ref{lemma M1}, the result follows. If $M=\varepsilon$, then by Lemma \ref{lemma M0} the claim also holds. Now, let us assume that $M=mN$, for some $m\in I$ and $N\in \mathcal{B}_{I}$. Let us take by inductive hypothesis that there exists $R\in \mathcal{B}_{I}$ such that $R(a\vee b)\leq N(a) \vee n(b)$. Thus from the monotonicity of $m$, (Cm) and the contractivity of $n$ we get
\[mR(a\vee b)\leq m(N(a) \vee n(b))\leq mN(a)\vee mn(b)\leq M(a)\vee n(b),\]
as desired. This concludes the proof.
\end{proof}

Now, we extend the previous result for arbitrary $I$-blocks. Its proof is similar to the one of previous Lemma, so we omit the details.

\begin{lemma}\label{lemma I-blocks}
Let \( \mathbf{A} \in \mathcal{R_{C}}(I) \) and $a,b \in A$. Then, for every $M,N\in \mathcal{B}_{I}$ there exists $Q\in \mathcal{B}_{I}$ such that 
\[Q(x\vee y)\leq M(x) \vee N(y). \]
\end{lemma}

\begin{lemma} \label{I filtro de supremo} Let \( \mathbf{A} \in \mathcal{R_{C}}(I) \) and \( a,b \in \mathbf{A} \). Then 
\[ \mathsf{Fg^{\mathbf{A}}}\left( a \vee b \right) = \mathsf{Fg^{\mathbf{A}}}(a) \cap \mathsf{Fg^{\mathbf{A}}}(b) \]
\end{lemma}
\begin{proof}
On the one hand, notice that from \( a,b \leq a \vee b \) it follows \( \mathsf{Fg^{\mathbf{A}}}(a \vee b) \subseteq \mathsf{Fg^{\mathbf{A}}}(a) \cap \mathsf{Fg^{\mathbf{A}}}(b) \). On the other hand, let \( x \in \mathsf{Fg^{\mathbf{A}}}(a) \cap \mathsf{Fg^{\mathbf{A}}}(b) \). Then, by Lemma \ref{generated filters R(I)}, there are blocks \( M,N\in \mathcal{B}_{I} \) such that \( M(a) \leq x \) and \( N(b) \leq x \), so $M(a) \vee N(b) \leq x.$
From Lemma \ref{lemma I-blocks} there exists \( Q\in \mathcal{B}_{I} \) such that $Q(a \vee b) \leq M(a) \vee N(b) \leq x$. Thus \( x \in \mathsf{Fg^{\mathbf{A}}}(a \vee b) \), as claimed.
\end{proof}

We are ready to prove the main result of this section.

\begin{theorem} 
In $\mathcal{R}_{c}(I)$, the variety generated by chains is \( \mathcal{R_{C}}(I) \). 
\end{theorem}
\begin{proof}
Note that \( \mathcal{C} \subseteq \mathcal{R_{C}}(I) \), thus it is immediate that \( \mathcal{V}(\mathcal{C}) \subseteq \mathcal{R_{C}}(I) \). In order to prove that \( \mathcal{R_{C}}(I) \subseteq \mathcal{V}(\mathcal{C}) \) it will be enough to prove that every subdirectly irreducible member of \( \mathcal{R_{C}}(I) \) belongs to \( \mathcal{C} \). Indeed, suppose that \( \mathbf{A}  \in \mathcal{R_{C}}(I) \) is  subdirectly irreducible and \( \mathbf{A} \) is not a chain. Then, there are \( a,b \in \mathbf{A} \) such that \( a \nleq b \) and \( b \nleq a \). From Lemma \ref{props riRigs} (3), the latter is equivalent to \( a \to b \neq 1 \) and \( b \to a \neq 1 \). Taking into account that \( \mathbf{A} \in \mathcal{R_{C}}(I) \), by (P) we have that \( 1 = (a \to b) \vee (b \to a) \). Hence by Lemmas \ref{Congruence I-Filters} and \ref{I filtro de supremo}, we get
\[ \Delta^{\mathbf{A}} = \theta_{\mathsf{Fg^{\mathbf{A}}}(a \to b)} \cap \theta_{\mathsf{Fg^{\mathbf{A}}}(b \to a)}, \]
wich is absurd. Therefore,  \( \mathbf{A} \) has to be a chain.
\end{proof}

\section{Compatible functions on $\mathcal{R}(I)$}\label{Compatible functions}

Extending a logic in a way that the logic remains unchanged by means of adding connectives is a problem widely studied along the literature. The typical example that one has in mind is given by the classical propositional calculus, in which if an axiomatic extension defines implicitly a new connective, then it must be deductively equivalent to a combination of classical connectives. In general this is not true for all logics, as the intuitionistic propositional calculus shows for the case of the \emph{succesor operator}. In \cite{CC2001} it was proved that when regarding such an operator, it is possible to define an axiomatic extension that defines implicitly a new connective which is not deductively equivalent to any combination of intuitionistic connectives.

As one can naturally expect, there is an algebraic counterpart of the implicit definability of connectives. It is called \emph{implicit definability by equations of new operations in varieties} \cite{CC2001} and we will focus on it along Section \ref{equationally defined op}. Nevertheless, due to the core of such an approach relies on the concept of \emph{compatible function} \cite{C20ican 04}, we will devote the contents of this section to the study of such functions in $\mathcal{R}(I)$. The results we will present will be obtained by means of the description of filter generation we provided in Lemma \ref{generated filters R(I)}. We recall some definitions below. 

\begin{definition} \label{cfc} 
Let $\V$ be a variety, $\mathbf{A}\in \V$ and $f:A^{n} \rightarrow A$ a map.
\begin{enumerate}
\item We say that $f$ is compatible with a congruence
$\theta$ of $\mathbf{A}$ if $(a_{i},b_{i}) \in \theta$ for $i=1, \ldots ,n$
implies $(f(a_{1}, \ldots ,a_{n}), f(b_{1}, \ldots ,b_{n})) \in
\theta $.
\item We say that $f$ is a compatible function of
$\mathbf{A}$ if it is compatible with all the congruences of $\mathbf{A}$.
\item We say that $f$ is $\V$-compatible (or just compatible, if the context is clear) if for every $\mathbf{A}\in \V$, $f$ is compatible with $\mathbf{A}$.
\end{enumerate}
\end{definition}

We stress that there are some useful connections between principal congruences and the compatible operations. Since these are well known in the litterature, we omit the details of the proof. 

\begin{remark}\label{rem: compatible functions}
Let $\V$ be a variety and $\mathbf{A} \in \V$.
\begin{enumerate}
\item Let $f:A \rightarrow A$ be a function. Then $f$ is compatible if and only if $(f(a),f(b))\in \mathsf{Cg}^{\mathbf{A}}(a,b)$ for every $a, b \in A$.
\item Let $f : A^{n} \rightarrow A$ be a function and $\vec{a} = (a_{1}, \ldots, a_{n}) \in A^{n}$. For $i = 1,\ldots, n$, we define unary functions $f^{\vec{a}}_{i}: A \rightarrow A$ by
\[
f^{\vec{a}}_{i}(x):= f(a_{1}, \ldots,a_{i-1}, x, a_{i+1}, \ldots, a_{n}).
\]
Then $f$ is compatible if and only if for every $\vec{a} \in A^{n}$ and every $i= 1,\ldots, n$ the functions $f^{\vec{a}}_{i} : A \to A$ are compatible.
\end{enumerate}
\end{remark}

The following Lemmas are immediate consequences of Lemma \ref{Congruence I-Filters} and Corollary \ref{cor:Generated filter intersection}.

\begin{lemma}\label{lem: Obvious lemma}
Let $\mathbf{A}\in \mathcal{R}(I)$, $\theta \in \mathsf{Con}(\mathbf{A})$ and $a,b\in A$. Then
$(a,b) \in \theta$ if and only $a\ast b \in 1/\theta$.
\end{lemma}

\begin{lemma}  \label{lem:Principal Congruences principal filters}
Let $\mathbf{A}\in \mathcal{R}(I)$ and $a,b\in A$. Then $(x,y) \in \mathsf{Cg}^{\mathbf{A}}(a,b)$ if and only if $M(a\ast b) \leq x\ast y$ for some $M\in \mathcal{B}_{I}$.
\end{lemma}

As an straightforward application of Lemma \ref{lem:Principal Congruences principal filters} and Remark \ref{rem: compatible functions} (1), we get a characterization of unary compatible functions.

\begin{lemma} \label{lem: 1-ary compatible functions}
Let $\mathbf{A}\in \mathcal{R}(I)$ and $f:A \rightarrow A$ a function. Then $f$ is compatible if and only if for every $a,b \in A$ there exists $M\in \mathcal{B}_{I}$ such that $M(a\ast b) \leq f(a) \ast f(b)$.
\end{lemma}

Now we characterize $k$-ary compatible functions.

\begin{theorem} \label{prop:k-ary compatible functions}
Let $\mathbf{A}\in \mathcal{R}(I)$ and let $f:A^{k} \rightarrow A$ be a function. The following conditions are equivalent:
\begin{enumerate}
\item $f$ is compatible.
\item For every $a_{1},\ldots, a_{k}, b_{1},\ldots,b_{k} \in A$ there exists $M_1,\ldots ,M_n\in \mathcal{B}_{I}$
such that \[M_1(a_1 \ast b_1)\cdot \ldots \cdot M_k(a_k \ast b_k) \leq f(a_{1}, \ldots a_{k}) \ast f(b_{1},\ldots ,b_{k}).\]
\end{enumerate}
\end{theorem}
\begin{proof}
We stress that by general properties of ririgs (Lemma \ref{props riRigs}), for $a,b,c\in A$ we have
 \begin{equation} \label{Term s}
(a\ast b) \cdot (b\ast c) \leq a\ast c.
\end{equation}
Now, suppose that $f$ is compatible and let $a_{1}, \ldots a_{k},
b_{1}, \ldots ,b_{k} \in A$. Consider $\vec{c}_{1}=(a_{1},a_{2}...,a_{k})$, $\vec{c}_{1}=(b_{1}, a_{2}...,a_{k})$, ..., $\vec{c}_{k}=(b_{1},b_{2}...,b_{k})$. Then by Remark \ref{rem: compatible functions} $f^{\vec{c}_{j}}_{i}$ is compatible, for every $i,j\in \{1,... k\}$. Since $(f^{\vec{c}_{i}}_{i}(a_{i}),f^{\vec{c}_{i}}_{i}(b_{i}))\in \mathsf{Cg}^{\mathbf{A}}(a_{i},b_{i})$ for every $1\leq i\leq k$, by Lemma \ref{lem: 1-ary compatible functions} there are $I$-blocks $M_{1}, \ldots ,M_{k}$ such that
\begin{displaymath}
\begin{array}{l}
M_{1}(a_{1} \ast b_{1}) \leq f(a_{1}, a_{2}, \ldots
,a_{k})\ast f(b_{1},a_{2}, \ldots ,a_{k})\\
M_{2}(a_{2} \ast b_{2}) \leq f(b_{1},a_{2},a_{3}, \ldots
,a_{k})\ast f(b_{1},b_{2},a_{3}, \ldots , a_{k})\\
\vdots \\
M_{k}(a_{k} \ast b_{k}) \leq f(b_{1}, \ldots
,b_{k-1},a_{k})\ast f(b_{1},b_{2}, \ldots ,b_{k}).
\end{array}
\end{displaymath}
Hence, by taking the product at both sides, and applying (\ref{Term s}) we get
\[
M_{1}(a_{1} \ast b_{1})\cdot \ldots \cdot M_{k}(a_{k} \ast b_{k}) \leq
f(a_{1}, \ldots a_{k}) \ast f(b_{1}, \ldots ,b_{k}).
\]
On the other hand, suppose that for every $a_{1},\ldots, a_{k}, b_{1}, \ldots, b_{k} \in A$ there exist $M_1,\ldots ,M_n\in \mathcal{B}_{I}$
such that $M_1(a_1 \ast b_1)\cdot \ldots \cdot M_k(a_k \ast b_k) \leq f(a_{1}, \ldots a_{k}) \ast f(b_{1},\ldots ,b_{k})$. By Lemma \ref{lem: Obvious lemma} we have $(a_{i} \ast b_{i},1) \in \theta$, so $(M_{i}(a_{i} \ast b_{i}),1) \in \theta $ for every $1\leq i \leq k$. Hence, $(M_{1}(a_{1} \ast b_{1})\cdot \ldots \cdot M_{k}(a_{k} \ast b_{k}),1)\in \theta$. Since $1/\theta$ is increasing, from our assumption we conclude 
$(f(a_{1}, \ldots a_{k}) \ast f(b_{1}, \ldots ,b_{k}),1)\in \theta$. Thus, again by
Lemma \ref{lem: Obvious lemma} we obtain that $(f(a_{1}, \ldots ,a_{k}),f(b_{1},
\ldots ,b_{k})) \in \theta$. Therefore, $f$ is compatible, as desired. 
\end{proof}

Now we study compatible functions in $\mathcal{R}(I_{\omega})$. We stress that Lemmas \ref{lem: 1-ary compatible functions finite} and \ref{lem:Principal Congruences principal filters finite} and Theorem \ref{prop:k-ary compatible functions finite} are a straightforward consequence of Lemmas \ref{lem: Obvious lemma} and \ref{generated filter finite case} and  its proofs are quite similar to the proofs of Lemmas \ref{lem:Principal Congruences principal filters} and \ref{lem: 1-ary compatible functions} and Theorem \ref{prop:k-ary compatible functions}, respectively. We left the details to the reader. 

\begin{lemma} \label{lem: 1-ary compatible functions finite}
Let $\mathbf{A}\in \mathcal{R}(I)$ and $f:A \rightarrow A$ a function. Then $f$ is compatible if and only if for every $a,b \in A$ there exists $l\in \mathbb{N}$ such that $\lambda^{l}(a\ast b) \leq f(a) \ast f(b)$.
\end{lemma}

\begin{lemma}  \label{lem:Principal Congruences principal filters finite}
Let $\mathbf{A}\in \mathcal{R}(I_{\omega})$ and $a,b\in A$. Then $(x,y) \in \mathsf{Cg}^{\mathbf{A}}(a,b)$ if and only if $\lambda^{l}(a\ast b) \leq x\ast y$ for some $l\in \mathbb{N}$.
\end{lemma}

\begin{theorem} \label{prop:k-ary compatible functions finite}
Let $\mathbf{A}\in \mathcal{R}(I_{\omega})$ and let $f:A^{k} \rightarrow A$ be a function. The following conditions are equivalent:
\begin{enumerate}
\item $f$ is compatible.
\item For every $a_{1},\ldots, a_{k}, b_{1},\ldots,b_{k} \in A$ there exists $l\geq 0$
such that \[\lambda^{l}(a_1 \ast b_1)\cdot \ldots \cdot \lambda^{l}(a_k \ast b_k) \leq f(a_{1}, \ldots a_{k}) \ast f(b_{1},\ldots ,b_{k}).\]
\end{enumerate}
\end{theorem}

We conclude this section by showing the locally affine completess of $\mathcal{R}(I_{\omega})$. A variety $\mathcal{V}$ is said to be \emph{affine complete} if for every $\mathbf{A}\in \mathcal{V}$, any compatible function of $\mathbf{A}$ is given by a polynomial of $\mathbf{A}$. It is called is \emph{locally affine complete} provided that for every $\mathbf{A}\in \mathcal{V}$, any compatible function is given by a polynomial on each finite subset of $A$.
\\

Let $\mathbf{A}\in \mathcal{R}(I_{\omega})$ and  let $f:A^{k} \rightarrow A$ be a compatible function. Let $\vec{a}
=(a_{1}, \ldots ,a_{k})$ and $\vec{b} = (b_{1}, \ldots ,b_{k})$ be elements
of $A^{k}$. Recall that from Theorem \ref{prop:k-ary compatible functions finite}, there exists a natural number $l$ satisfying condition (2) associated to $(\vec{a},\vec{b})$. In what follows, we write $n(\vec{a},\vec{b})$ for such a number $l$. Then, if $B$ is a finite subset of $A^{k}$ and $\vec{a}\in B$ is fixed, we have a finite family of natural numbers, namely, $\lbrace n(\vec{a},\vec{x}): \vec{x}\in B \rbrace$. Let $n_{\vec{a}}$ be the maximum of this family. The following theorem shows that every compatible function on finite subsets can be written as a join of some suitable elements.

\begin{theorem} \label{lac}
Let $\mathbf{A}\in \mathcal{R}(I_{\omega})$, let $f:A^{k} \rightarrow A$ be a compatible function. If $B$ is a finite subset of $A^{k}$ and $\vec{x}\in B$, then $f(x_{1},...,x_{n})=\bigvee T_{\vec{x}}$, where
\[
T_{\vec{x}} = \lbrace \prod_{i=1}^{k} \lambda^{n_{\vec{a}}}(a_{i} \ast x_{i}) \cdot f(a_{1},...,a_{k}):
\vec{a}\in B\rbrace.
\]
\end{theorem}

\begin{proof}
Let $\vec{x}\in B$. Then, since $f$ is compatible by assumption, for every $\vec{a}\in B$ and Theorem \ref{prop:k-ary compatible functions finite} we have
\[\prod_{i=1}^{k} \lambda^{n_{\vec{a}}}(a_{i} \ast x_{i}) \leq f(\vec{a})\ast f(\vec{x})\leq f(\vec{a})\rightarrow f(\vec{x}).\]
Thus \[\prod_{i=1}^{k} \lambda^{n_{\vec{a}}}(a_{i} \ast x_{i})\cdot f(\vec{a})\leq f(\vec{x}).\]
So $f(\vec{x})$ is an upper bound of $T_{\vec{x}}$. Finally, notice that because $x_{i}\rightarrow x_{i}=1$, then by Theorem \ref{prop:k-ary compatible functions finite} we may conclude
\[\prod_{i=1}^{k} \lambda^{n_{\vec{x}}}(x_{i} \ast x_{i}) \cdot f(\vec{x})= f(\vec{x})\]
Therefore $f(\vec{x})\in T_{\vec{x}}$. This concludes the proof.
\end{proof}

As an immediate consequence of Theorem \ref{lac}, we get:

\begin{corollary}
$\mathcal{R}(I_{\omega})$ is locally affine complete.
\end{corollary}

\subsection{Equationally defined compatible operations on $\mathcal{R}(I)$}\label{equationally defined op}

Let $\Sigma(f)$ be a set of equations in the signature of $\mathcal{R}(I)$ augmented by a $n$-ary function symbol. We say that $\Sigma(f)$ \emph{defines an implicit operation of $\mathcal{R}(I)$} if for every $\mathbf{A}\in \mathcal{R}(I)$, there exists at most one function $f^{\mathbf{A}}:A^{n}\rightarrow A$ such that $(\mathbf{A}, f^{\mathbf{A}})\models \Sigma(f)$. We also say that $f$ is an \emph{implicit compatible operation} if for every $\mathbf{A}\in \mathcal{R}(I)$, the function $f^{\mathbf{A}}$ is compatible.
\\

Notice that due to the cardinality of $I$, the set $\Sigma(f)$ is at least countable. Therefore, it may be interesting to study whether an equationally defined compatible operation on $\mathcal{R}(I)$ can be explicitly defined by a term in the language of $\mathcal{R}(I)$ and implicitly defined by a finite set of equations in the language of $\mathcal{R}(I)$. To achieve this goal and inspired in \cite{CV2016}, we introduce some notions first.

Let $f \in \mathcal{L}$ be a function symbol, and let $\mathcal{K}$ be a class of $\mathcal{L}$-structures.  We write $f=t|_{\mathcal{\varphi}}$ to express that there exist an $\mathcal{L}$-term $t(\vec{x})$ and a first order $\mathcal{L}$-formula $\varphi(\vec{x})$ such that the following conditions hold: 
\begin{enumerate}
\item $\mathcal{K}\models \varphi(\vec{x})$,
\item $f^{\mathbf{A}}(\vec{a})=t^{\mathbf{A}}(\vec{a})$ if $\mathbf{A}\models \varphi(\vec{a})$, for every $\mathbf{A}\in \mathcal{K}$ and $\vec{a}\in A^{n}$.
\end{enumerate}
We say that a term $t(\vec{x})$ \emph{represents $f$ in $\mathcal{K}$} if $f^{\mathbf{A}}(\vec{a})=t^{\mathbf{A}}(\vec{a})$ for every $\mathbf{A}\in \mathcal{K}$ and $\vec{a}\in A^{n}$.

\begin{theorem}
Let $\Sigma(f)$ be a set of equations defining implicitly an operation $f$ in $\mathcal{R}(I)$ and consider 
\[\mathcal{M}_{f}=\{ \mathbf{A}\in \mathcal{R}(I)\colon (\mathbf{A}, f^{\mathbf{A}}) \models \Sigma(f)\; \text{for some}\;f^{\mathbf{A}}:A^{n}\rightarrow A\}.\]
Then, the following are equivalent:
\begin{enumerate}
\item $f$ is representable by a unique $n$-ary term $t$ in the language of $\mathcal{R}(I)$ and $f=t|_{\mathcal{\varphi}}$ in $\mathcal{M}_{f}$ some formula $(\bigwedge p=q)$-formula $\varphi(\vec{x})$ in the language of $\mathcal{R}(I)$.
\item The following conditions hold:
\begin{itemize}
\item[(a)] $\mathcal{M}_{f}$ is closed under subalgebras. 
\item[(b)] For every $\mathbf{A}\in \mathcal{M}_{f}$, $f^{\mathbf{A}}$ is compatible.
\item[(c)] For all $\mathbf{A}, \mathbf{B}\in \mathcal{M}_{f}$, all $\mathbf{A}_{0}\leq \mathbf{A}$, $\mathbf{B}_{0}\leq \mathbf{B}$, all homomorphisms $\sigma: \mathbf{A}_{0}\rightarrow \mathbf{B}_{0}$, and $a_{0},..., a_{n}\in A$ with $f^{\mathbf{A}}(\vec{a})\in B_{0}$, we have 
\[\sigma(f^{\mathbf{A}}(\vec{a}))=f^{\mathbf{A}}(\sigma(a_{1}),\ldots, \sigma(a_{n})).\]

\end{itemize}
\end{enumerate}

\end{theorem}
\begin{proof}
Observe that from Lemma 5 of \cite{C20ican 04}, $f$ is representable by a unique $n$-ary term $t$ in the language of $\mathcal{R}(I)$ if and only if $(a)$ and $(b)$ hold. For the last part, let $\mathcal{L}'=\mathcal{L} \cup \{f\}$, where $f$ is an $n$-ary function symbol and let us consider the following class of $\mathcal{L}'$-structures
\[\mathcal{K}=\{(\mathbf{A}, f^{\mathbf{A}})\colon  \mathbf{A}\in \mathcal{M}_{f}\}. \]
Observe that $\mathcal{K}$ is a first order class by Beth's definability theorem and it is closed by products. Hence, by Theorem 5.3 of \cite{CV2016}, there is a $(\bigwedge p=q)$-formula $\varphi(\vec{x})$ in the language of $\mathcal{R}(I)$ which defines $f$ in $\mathcal{K}$ if and only if for every $(\mathbf{A},f^{\mathbf{A}}), (\mathbf{B},f^{\mathbf{B}})\in \mathcal{K}$, $\mathbf{A}_{0}\leq \mathbf{A}$ and $\mathbf{B}_{0}\leq \mathbf{B}$ and all homomorphism $\sigma:\mathbf{A}_{0}\rightarrow \mathbf{B}_{0}$, then $\sigma:(\mathbf{A}_{0},f^{\mathbf{A}})\rightarrow (\mathbf{B}_{0},f^{\mathbf{B}})$ is a homomorphism. It is no hard to see that the latter is equivalent to $(c)$.
\end{proof}

\section{The logic associated to \( \mathcal{R}(I) \)}\label{logic R(I)}

In this section we introduce a  Hilbert-style calculus \( \mathcal{H} \) and afterwards we define a deductive system \( \mathcal{S}_{\mathcal{H}} \). We prove that \( \mathcal{S}_{\mathcal{H}} \) is Block-Pigozzi algebraizable, with equivalent variety semantics \( \mathcal{R}(I) \) and structural transformers \( \tau(\varphi):= \{ \varphi \approx 1 \} \) and \( \rho(\varphi \approx \psi):= \{ \varphi \to \psi, \psi \to \varphi \} \).  
\\

Let \( I \) be an arbitrary non empty set of symbols such that \( \mathcal{L} \cap I = \emptyset \) and let \( \mathcal{L}(I)=\{ \vee, \to, \cdot, \{m\}_{m \in I} ,0,1 \} \) be the signature of I-modal ririgs. Since there is no risk of confusion we write \( \mathbf{Fm} \) for the absolutely free algebra of formulas in the language \( \mathcal{L}(I) \). We define a constant $\top$ by \( \bot \to \bot \). The expression \( \varphi \leftrightarrow \psi \) is a shorthand to \( \left( \varphi \to \psi \right) \) and \( \left( \psi \to \varphi \right) \). Now we present the following Hilbert-style calculus, in where the axioms ($m$11) and ($m$12), and the rule ($m$Nec) range over $I$.

\vspace{0.2cm}
\textbf{Hilbert axioms}:

\begin{itemize}
\item[(1)] \(  \varphi \to \varphi \)
\item[(2)] \(  (\varphi \to \psi) \to [(\psi \to \chi) \to (\varphi \to \chi)] \)
\item[(3)] \( \varphi \cdot \psi \to \varphi \)
\item[(4)] \( \varphi \cdot \psi \to \psi \cdot \varphi \)
\item[(5)] \( [(\varphi \cdot \psi) \to \chi] \to [\psi \to (\varphi \to \chi)] \)
\item[(6)] \( [\psi \to (\varphi \to \chi)] \to [(\varphi \cdot \psi) \to \chi] \)
\item[(7)] \( \varphi \to (\varphi \vee \psi) \)
\item[(8)] \( \psi \to (\varphi \vee \psi) \)
\item[(9)] \( \varphi \cdot (\psi \vee \chi) \to [(\varphi \cdot \psi) \vee (\varphi \cdot \chi)] \)
\item[(10)] \( \bot \to \varphi \)
\item[($m$11)] \( m (\top) \leftrightarrow \top \), for all \( m \in I \)
\item[($m$12)] \( m(\varphi \to \psi) \to (m(\varphi) \to m(\psi)) \), for all \( m \in I \)
\end{itemize}

\vspace{0.2cm}

\textbf{Hilbert rules}

\begin{itemize}
    \item[(MP)] \( \{ \varphi \to \psi, \varphi \} \rhd \psi \)
    \item[($m$Nec)] \( \{ \varphi \} \rhd m(\varphi) \) for all \( m \in I \)
    \item[(\( \rhd \vee \))] \( \{ \varphi \to \chi, \psi \to \chi \} \rhd \left( \varphi \vee \psi \right) \to \chi \)
\end{itemize}

An \( \mathcal{H} \)- deduction (or proof) of \( \varphi \) from the set of formulas \( \{ \gamma_{1},..,\gamma_{k} \} \) is a finite sequence of formulas \( \varphi_{1},..., \varphi_{n}  \) such that 
\begin{itemize}
    \item[(1)] \( \varphi_{1} \) is obtained by axioms or is a formula of \( \{ \gamma_{1},..,\gamma_{k} \} \),
    \item[(2)] \( \varphi_{i} \) is obtained by applications of axioms and rules to the set \( \{ \gamma_{1},..,\gamma_{k}, \varphi_{1},..,\varphi_{i-1} \} \),
    \item[(3)] \( \varphi_{n} = \varphi \).
\end{itemize}

If there is a deduction of \( \varphi \) from the set of formulas \( \Gamma \) then we say that \( \varphi \) is a \emph{consequence of \( \Gamma \)}. Also we say that \emph{\( \varphi \) follows from \( \Gamma \)}.

\begin{definition} The deductive sistem \( S_{\mathcal{H}} = \langle \mathbf{Fm},\vdash_{S_{\mathcal{H}}} \rangle \) is defined as follows:  for each set of formulas \( \Gamma \) and \( \varphi \in \mathbf{Fm} \) 
\[ \Gamma \vdash_{S_{\mathcal{H}}} \varphi \text{ if and only if there are } \gamma_{1},...,\gamma_{n} \in \Gamma  \text{ such that } \varphi \text{ follows from } \{\gamma_{1},...,\gamma_{n}\}. \]
\end{definition}

A formula \( \varphi \) is said to be a \emph{theorem} of \( S_{\mathcal{H}} \) provided that \( \emptyset \vdash_{S_{\mathcal{H}}} \varphi \). Besides we say that a set of formulas \( \Gamma \) is a \emph{theory} of \( S_{\mathcal{H}} \) if \( \Gamma \vdash_{S_{\mathcal{H}}} \varphi \) implies \( \varphi \in \Gamma \). Some simple computations shows the following:

\begin{lemma} \label{teoremas de s} The following formulas are theorems of \( S_{\mathcal{H}} \):
\begin{itemize}
    \item[(1)] \( \varphi \to \top \),
    \item[(2)] \( \varphi \to ( \psi \to \varphi ) \),
    \item[(3)] \( (\varphi \to \psi) \to (\varphi \cdot \chi) \to (\psi \cdot \chi) \),
    \item[(4)] \( (\varphi \cdot \psi) \cdot \chi \to \varphi \cdot (\psi \cdot \chi) \).
\end{itemize}
\end{lemma}

\begin{remark} \label{comparasion con logica van alten} 
Notice that by the distributive axiom and the weaking axiom, the \( I \) free reduct of \( S_{\mathcal{H}} \) is an extension of the \( \{ \vee,\to,\cdot\} \)-fragment of the deductive system \( \mathbf{L} \) given in \cite{VAR}. The \( \{ \vee,\to,\cdot\} \)-fragment of \( \mathbf{L} \) is algebraizable with equivalent algebraic semantics given by the \( \{\vee,\to\} \)-subreduct of the variety \( \mathbf{Res} \) of residuated lattices. Moreover, $\mathcal{H}$ is an implicative logic (for details, see Chapter II of \cite{F}).
\end{remark}

By the previous remark and taking into account that the logic \( S_{\mathcal{H}} \) is a expansion of \( \mathbf{L} \) wich is algebraizable, we shall prove that \( S_{\mathcal{H}} \) is algebraizable too. To this end, let \( \Gamma \) be a theory of \( S_{\mathcal{H}} \). We define the relation $\Omega(\Gamma)$ by
\[ \Omega(\Gamma)= \{ (\varphi,\psi) \in \mathbf{Fm} \times \mathbf{Fm} \colon \{ \varphi \to \psi, \psi \to \varphi \} \subseteq \Gamma \}. \] 

Straightforward computations show that \( \Omega(\Gamma) \) is a congruence on \( \mathbf{Fm} \) compatible with \( \Gamma \), in the sense that \( \gamma \in \Gamma \) if and only if \( \left(\gamma,\top \right) \in \Omega\left( \Gamma \right) \). We write \( \varphi/\Omega(\Gamma) \) for the equivalence class of \( \varphi \) and we use \( \varphi \equiv \psi \) to denote that \( \left( \varphi,\psi \right) \in \Omega(\Gamma) \). In what follows, we consider the algebra \( \langle \textbf{Fm}/\Omega(\Gamma), \vee, \to ,\cdot, \{m\}_{m \in I} ,0 , 1 \rangle \) of type \( \mathcal{L}(I) \) with \( 1 := \top / \Omega(\Gamma) \) and \( 0 := \bot/ \Omega(\Gamma) \).

\begin{lemma} \label{algebra LT} The algebra \( \langle \textbf{Fm}/\Omega(\Gamma), \vee, \to ,\cdot, \{m\}_{m \in I} ,0 , 1 \rangle \) is an I-modal ririg. Moreover, the relation \( \varphi/\Omega(\Gamma) \leq_{\Gamma} \psi/\Omega(\Gamma) \) if and only if \( \varphi \to \psi \in \Gamma \) is a partial order on \( \mathbf{Fm}/\Omega(\Gamma) \), with \( 1/\Omega(\Gamma) \) as the largest element.
\end{lemma}
\begin{proof}
Notice that by axioms (7),(8),(9) and the inference rule \( \left( \rhd \vee \right) \) we have that  \( \langle \mathbf{Fm}/ \Omega(\Gamma), \vee, 0 \rangle \) is a commutative monoid. Besides, by axioms (3), (4) and theorems (3),(4) of Lemma \ref{teoremas de s} we have that   \( \langle \mathbf{Fm}/ \Omega(\Gamma), \cdot, 1 \rangle \) is a monoid too. Furthermore, by axiom (10) we have that the product \( \cdot \) distributes over joins.

Let us see that \( \leq_{\Gamma} \) is a order. By axiom (1) and (2) we have that \( \leq_{\Gamma} \) is a reflexive and transitive relation. The antisymmetry is immediate and by (1) of Lemma \ref{teoremas de s} it is clear that \( 1/\Omega(\Gamma) \) is the largest element of the poset $\langle \textbf{Fm}/\Omega(\Gamma),\leq_{\Gamma}\rangle$.

Observe that the residuation of the pair \( (\cdot,\to) \) follows by axioms (5) and (6). Indeed, suppose that \( \varphi/ \Omega(\Gamma) \cdot \psi/ \Omega(\Gamma) \leq_{\Gamma} \chi/\Omega(\Gamma) \), i.e.  \( (\varphi \cdot \psi) \to \chi \in \Gamma \). By axiom (5) we have that \(  [(\varphi \cdot \psi) \to \chi] \to [\psi \to (\varphi \to \chi)] \in \Gamma \), so by assumption and (MP) we have that \( \Gamma \vdash_{S_{\mathcal{H}}} \psi \to (\varphi \to \chi) \). Thus \( \psi \to (\varphi \to \chi) \in \Gamma \) so \( \psi/\Omega(\Gamma) \leq \varphi/\Omega(\Gamma) \to \chi/\Omega(\Gamma) \). Now, suppose that \( \psi/\Omega(\Gamma) \leq_{\Gamma} \varphi/\Omega(\Gamma) \to \chi/\Omega(\Gamma) \), then  \( \psi \to (\varphi \to \chi) \in \Gamma \). By axiom (6) it follows \( [\psi \to (\varphi \to \chi)] \to [(\varphi \cdot \psi) \to \chi] \in \Gamma  \). Hence by (MP) we have that \( (\varphi \cdot \psi) \to \chi \in \Gamma \) and \( \varphi/\Omega(\Gamma) \cdot \psi/\Omega(\Gamma) \leq_{\Gamma} \chi/\Omega(\Gamma) \), as claimed.

Finally, by axioms (11$m$) and (12$m$) it is clear that each \( m \in I \)  is a modal operator. Therefore the quotient algebra \( \mathbf{Fm}/\Omega(\Gamma) \) is an I-modal ririg. This concludes the proof.
\end{proof}
 
Let us consider the structural transformers \( \tau \colon \mathrm{Fm} \to \mathcal{P}(\mathrm{Eq}) \) and \( \rho \colon \mathrm{Eq} \to \mathcal{P}(\mathrm{Fm}) \), defined by \( \tau(\varphi)= \{ \varphi \approx 1 \} \) and \( \rho(\varphi \approx \psi) := \{ \varphi \to \psi, \psi \to \varphi \} \), respectively. Taking into account Lemma \ref{algebra LT} we have the following:

\begin{theorem} \label{algebraizacion} The deductive system \( S_{\mathcal{H}} \) is algebraizable with equivalent algebraic semantic given by the variety \( \mathcal{R}(I) \) and the structural transformers \( \tau(\varphi)=\{ \varphi \approx 1 \} \) and  \( \rho(\varphi \approx \psi) := \{ \varphi \to \psi, \psi \to \varphi \} \).

\end{theorem}
\begin{proof} We shall prove (1) and (2) of remmark \ref{remark algebraizacion}.

(1) \( (\Rightarrow) \) Suppose that \( \Gamma \vdash_{S_{\mathcal{H}}} \varphi \). We shall prove that \( \tau[\Gamma] \vDash_{\mathcal{R}(I)} \tau(\varphi) \). It will be enough to prove that \( \mathcal{R}(I) \vDash \varphi \approx 1 \) holds for each \( \varphi \) axiom of \( \mathcal{H} \) and that \(  \{ \gamma \approx 1 \colon \gamma \in  \Gamma \} \vDash_{\mathcal{R}(I)}  \varphi \approx 1 \) holds for each \( \Gamma \rhd \varphi \) Hilbert rule of \( \mathcal{H} \). Some simple computations shows this fact.

\( (\Leftarrow) \) We proceed by the contrapositive. Suppose that \( \Gamma \not \vdash_{S_{\mathcal{H}}} \varphi \) and let us consider the theory 
\[ \Delta = \{ \psi \in \mathbf{Fm} \colon \Gamma \vdash_{S_{\mathcal{H}}} \psi \}. \]
It is clear that \( \Gamma \subseteq \Delta \) and by Lemma \ref{algebra LT}, we have that the algebra 
\[ \left\langle \mathbf{Fm}/\Omega(\Delta), \vee,\to,\cdot,\{m \}_{m \in I},0,1 \right\rangle \]
is a member of \( \mathcal{R}(I) \). If we consider the canonical homomorphism \( \pi \colon \mathbf{Fm} \to \mathbf{Fm / \Omega(\Gamma)} \), then we get that \( \left( \mathbf{ Fm / \Omega(\Gamma) },\pi \right) \vDash \gamma \approx 1 \) for each \( \gamma \in \Gamma \) but \( \left( \mathbf{ Fm / \Omega(\Gamma) },\pi \right) \not \vDash \varphi \approx 1 \). Thus we have that \( \left( \mathbf{Fm}/\Omega(\Delta),\pi \right) \vDash \gamma \approx 1 \) for each \( \gamma \in \Gamma \) but \( \left( \mathbf{Fm}/\Omega(\Delta),\pi \right) \not \vDash \varphi \approx 1 \). Therefore \( \tau[\Gamma] \not \vDash_{\mathcal{R}(I)} \tau(\varphi) \).

(2) From Lemma \ref{props riRigs} it follows that for every I-modal ririg \( \mathbf{A} \) and \( a,b \in \mathbf{A} \) it holds \( a = b \) if and only if \( a \to b = 1 \) and \( b \to a = 1 \).
\end{proof}

\subsection{Some deduction theorems}

In this section we prove that the logic $\mathcal{S}_{\mathcal{H}}$ enjoys of a meta-logical property called the local deduction-detachment theorem. To do so, we will make use of the results we obtained along Sections \ref{I-modal ririgs} and \ref{The finite case}. We recall that a logic $\mathbf{L}$ has the \textit{local deduction-detachment theorem} (or \emph{LDDT}) if there exists a family $\{d_j(p,q)\colon j\in J\}$ of sets $d_j(p,q)$ of formulas in at most two variables such that for every set $\Gamma\cup \{\varphi,\psi\}$ of formulas in the language of $\mathbf{L}$:
\begin{displaymath}
\begin{array}{ccc}
\Gamma, \varphi \vdash_{\mathbf{L}} \psi &  \Longleftrightarrow & \Gamma  \vdash_{\mathbf{L}} d_{j}(\varphi,\psi)\; \text{for some}\; j\in J.
\end{array}
\end{displaymath}

If $\mathcal{V}$ is a variety, we denote by $\mathbf{F_{\mathcal{V}}}(X)$ the $\mathcal{V}$-free algebra over $X$. Moreover, if $\varphi$ is a formula, we write $\bar{\varphi}$ for the image of $\varphi$ under the natural map ${\bf Fm}(X)\to \mathbf{F}_{\mathcal{V}}(X)$ from the term algebra $\mathbf{Fm}(X)$ over $X$ onto $\mathbf{F}_{\mathcal{V}}(X)$. If $\Gamma$ is a set of formulas, we also denote by $\bar{\Gamma}$ the set $\{\bar{\varphi} : \varphi\in\Gamma\}$. The following is a technical result which is essentially restatement of Lemma 2 of \cite{MMT2014}.

\begin{lemma}\label{lem: Technical lemma}
Let $\Theta\cup \{\varphi\eq\psi\}$ be a set of equations in the language of $\mathcal{V}$, and let $X$ be the set of variables occurring in $\Theta\cup\{\varphi\eq\psi\}$. Then the following are equivalent:
\begin{enumerate}
\item $\Theta \models_{\mathcal{V}} \varphi \eq \psi$.
\item $(\bar{\varphi}, \bar{\psi})\in \bigvee_{\epsilon\eq\delta\in\Theta} \mathsf{Cg}^{\mathbf{F}_{\mathcal{V}}(X)}(\bar{\epsilon},\bar{\delta})$.
\end{enumerate}
\end{lemma}

\begin{theorem}\label{theo: R(I) satisfies parametrized LDDT}
Let $I$ be a set of unary connectives with $I\cap\mathcal{L}=\emptyset$, and suppose that $\mathbf{L}$ is an axiomatic extension of $S_{\mathcal{H}}$ that is algebraized by the subvariety $\mathcal{V}$ of $\mathcal{R}(I)$. Further, let $\Gamma\cup \Delta \cup \{\psi\} \subseteq Fm_{\mathcal{L}(I)}$. Then $\Gamma, \Delta \vdash_{\mathbf{L}} \psi$ if and only if for some $n\geq 0$ there exist I-blocks $M_{1},\ldots,M_{n}$ and $\psi_{1},\ldots,\psi_{n}\in \Delta$ such that $\Gamma \vdash_{\mathbf{L}} \prod_{j=1}^{n}M_{j}(\psi_{j})\rightarrow \psi$.
\end{theorem}
\begin{proof}
Let \( \Gamma \) and \( \Delta \) be sets of formulas and \( \varphi \) be a formula. If we assume that \( \Gamma \cup \Delta \vdash_{\mathcal{S}_{\mathcal{H}}} \varphi \) then there are \( \psi_{1},..,\psi_{n} \) finite subset of formulas of \( \Gamma \cup \Delta \) such that \( \{ \psi_{1}...,\psi_{n} \} \vdash_{\mathcal{S}_{\mathcal{H}}} \varphi \). Recall that from Theorem \ref{algebraizacion} we have that \( \{ \psi_{1}...,\psi_{n} \} \vdash_{\mathcal{S}_{\mathcal{H}}} \varphi \) if only if \( \{ \psi_{i} \approx 1 \colon 1 \leq i \leq n \} \vDash_{\mathcal{R}(I)} \varphi \approx 1 \). Therefore, due to Lemmas \ref{lem: Technical lemma} and \ref{cor:Generated filter intersection} (2), it is the case that \( \{ \psi_{i} \approx 1 \colon 1 \leq i \leq n \} \vDash_{\mathcal{R}(I)} \varphi \approx 1 \)  if and only if \( \left( \varphi,1 \right) \in \bigvee_{i=1}^{n} \mathsf{Cg}^{\mathbf{F}_{\mathrm{V}}(X)} \left( \psi_{i},1 \right) \) if and only if \( \varphi \in \mathsf{Fg}^{\mathbf{F}_{\mathrm{V}}(X)}\left( \{ \psi_{1},..,\psi_{n}  \} \right) \). Thus there are \( M_{1},..,M_{n}\in \mathcal{B}_{I} \) such that 
\[ \prod_{i=1}^{n} M_{i}\left( \psi_{i} \right) \leq \varphi. \]
Let \( C = \{ i \in \{1,...,n\} \colon \psi_{i} \in \Gamma \} \) and \( D = \{1,..,n\} \setminus C \), then from the commutativity of $\cdot$ we have that 
\[ \prod_{i \in C} M_{i}\left( \psi_{i} \right) \cdot \prod_{j \in D} M_{j}\left( \psi_{j} \right) \leq \varphi. \]
Hence, by the residuation law we have that 
\[ \prod_{i \in C} M_{i}\left( \psi_{i} \right) \leq \prod_{j \in D} M_{j}\left( \psi_{j} \right) \to \varphi. \]
Which, by Lemma \ref{generated filters R(I)}, means that \( \prod_{j \in D} M_{j}\left( \psi_{j} \right) \to \varphi \in  \mathsf{Fg}^{\mathbf{F}_{\mathrm{V}}(X)} \left( \{ \psi_{1},..,\psi_{n} \} \right) \). Finally, by Lemma \ref{lem: Technical lemma} and Theorem \ref{algebraizacion} we have that \( \Gamma \vdash_{\mathcal{S}_{\mathcal{H}}} \prod_{j \in D} M_{j}\left( \psi_{j} \right) \to \varphi \). The proof of the converse is analogue. This concludes the proof.
\end{proof}

The following theorem is the version of Theorem \ref{theo: R(I) satisfies parametrized LDDT} when considering I finite. Its proof uses the same arguments we applied on the proof the above Theorem with the difference that instead of Lemma \ref{generated filters R(I)} we employ Lemma \ref{generated filter finite case}. We leave the details to the reader.

\begin{theorem}\label{theo: R(If) satisfies parametrized LDDT}
Let $I$ be a finite set of unary connectives with $I\cap\mathcal{L}=\emptyset$, and suppose that $\mathbf{L}$ is an axiomatic extension of $S_{\mathcal{H}}$ that is algebraized by the subvariety $\mathcal{V}$ of $\mathcal{R}(I_{\omega})$. Further, let $\Gamma\cup \Delta \cup \{\psi\} \subseteq Fm_{\mathcal{L}(I)}$. Then $\Gamma, \Delta \vdash_{\mathbf{L}} \psi$ if and only if for some $n\geq 0$ there exist $l\geq 0$ and $\psi_{1},\ldots,\psi_{n}\in \Delta$ such that $\Gamma \vdash_{\mathbf{L}} \prod_{j=1}^{n}\lambda^{l}(\psi_{j})\rightarrow \psi$.
\end{theorem}

Observe that if we take $\Delta = \{\varphi\}$ and $d_{M} (p, q) = M(p) \rightarrow q$ for $M \in \mathcal{B}_{I}$, from Theorem \ref{theo: R(I) satisfies parametrized LDDT} we obtain:

\begin{corollary}\label{LDDT}
The logic $\mathcal{S}_{\mathcal{H}}$ has the LDDT.
\end{corollary}

We stress that when $I$ is finite, by taking $\Delta = \{\varphi\}$ and considering $d_{l} (p, q) = \lambda^{l}(p) \rightarrow q$ for $l \in \mathbb{N}$ we are also able to conclude that the logic $\mathcal{S}_{\mathcal{H}}$ has the LDDT.
\\

This section concludes with a result concerning the variety $\mathcal{R}(I)$ which is a consequence from the LDDT of $\mathcal{S}_{\mathcal{H}}$ obtained in Theorem \ref{theo: R(I) satisfies parametrized LDDT}.

We recall that an algebra $\mathbf{B}$ has the \emph{congruence extension property} (or \emph{CEP}) if for every subalgebra $\mathbf{A}$ of $\mathbf{B}$ and for any $\theta \in \mathsf{Con}(\mathbf{A})$, there exists $\xi \in \mathsf{Con}(\mathbf{B})$ such that $\xi\cap A^{2} = \theta$. A variety $\mathcal{V}$ is said to have the congruence extension property if each $\mathbf{B}\in\mathcal{V}$ does. It is well known \cite[Corollary 5.3]{BP} that if $\mathbf{L}$ is an algebraizable logic with equivalent variety semantics $\mathcal{V}$, then $\mathbf{L}$ has the LDDT if and only if $\mathcal{V}$ has the CEP. Hence, from Theorem \ref{theo: R(I) satisfies parametrized LDDT} we may conclude:

\begin{corollary}\label{theo: S4tL satisfies LDDT}
The variety $\mathcal{R}(I)$ has the CEP.
\end{corollary}

\bibliographystyle{plain}

\end{document}